\documentclass{article}%
\usepackage{graphicx}
\usepackage{amsmath,amssymb,amsfonts}%
\usepackage{theorem}%
\usepackage{color}%
\usepackage{hyperref}
\usepackage[font={small, it}]{caption}

\usepackage{caption}
\usepackage{subcaption}
\usepackage{float}
\usepackage{tikz-cd}
\tikzcdset{scale cd/.style={every label/.append style={scale=#1},
    cells={nodes={scale=#1}}}}
\usepackage{forest}
\usepackage{tikz-qtree}
\usetikzlibrary{trees}    
\theoremstyle{change}%
\usepackage{cite}

\newtheorem{definition}{Definition:}[section]%
\newtheorem{proposition}[definition]{Proposition:}%
\newtheorem{theorem}[definition]{Theorem:}%
\newtheorem{lemma}[definition]{Lemma:}%
{\theorembodyfont{\rmfamily}\newtheorem{remark}[definition]{Remark:}}%
{\theorembodyfont{\rmfamily}}%

\newenvironment{proof}
{{\bf Proof:}}
{\qquad \hspace*{\fill} $\Box$}%

\newcommand{\fh}{\mathfrak{h}}%
\newcommand{\tr}{\operatorname{tr}}%

\newcommand{\inner}{\operatorname{int}}%
\newcommand{\rme}{\mathrm{e}}%

\newcommand{\CC}{\mathcal{C}}%
\newcommand{\OC}{\mathcal{O}}%
\newcommand{\UC}{\mathcal{U}}%
\newcommand{\DC}{\mathcal{D}}%
\newcommand{\PC}{\mathcal{P}}

\newcommand{\HB}{\mathbb{H}}%
\newcommand{\R}{\mathbb{R}}%

\begin{document}

\title{A complete classification of control sets for singular linear control systems on the Heisenberg group}
\author{Adriano Da Silva\thanks{Supported by ``Fortalecimiento Grupos de Investigación UTA N$\textsuperscript{\underline{o}}$ 8802-25''} \\
Departamento de Matemática\\
Universidad de Tarapac\'{a}, Iquique, Chile \and
Okan Duman\\
Department of Mathematics \\
Yildiz Technical University - Istanbul, Turkey.\\
\and Anderson Felipe Penagos Rojas \\
Instituto de Matem\'{a}tica\\
Universidade Estadual de Campinas, Brazil\\
}
\date{\today }
\maketitle

\begin{abstract}
In this paper, we investigate the control sets of linear control systems on the Heisenberg group associated with singular derivations. Under the Lie algebra rank condition, we provide a complete characterization of these sets by analyzing the trace and determinant of an associated $2 \times 2$ submatrix.
\end{abstract}

\section{Introduction}
Understanding the dynamics of control systems is a key issue in modern mathematics and engineering. Traditional studies have largely focused on linear control systems (abbrev. LCSs) in Euclidean spaces, where many techniques have been well developed and have direct physical applications (see, for example, references \cite{FCWK, leit, markus, shell}). However, when the underlying space possesses non-trivial geometry, for instance, when the state space is a manifold or a Lie group; the dynamics can exhibit behaviours that have no analogue in the Euclidean setting.

In recent decades, there has been growing interest in studying LCSs on Lie groups because of their deep connections with differential geometry, representation theory, and nonlinear dynamics. The first significant step in this direction was made by L. Markus in \cite{markus}, who extended the framework of LCSs to matrix groups. Later, V. Ayala and J. Tirao \cite{VAJT} generalized this concept to arbitrary Lie groups, establishing a unifying geometric perspective. A further motivation was provided by P. Jouan in \cite{JPh1}, who showed that every control-affine system with complete vector fields generating a finite-dimensional Lie algebra is equivalent to a LCS on a Lie group or a homogeneous space. These results demonstrate that the study of control systems on Lie groups is not merely an abstract generalization but rather a natural extension of the Euclidean theory.  In the absence of a comprehensive global theory with general hypotheses, several studies of LCSs on various state spaces, including nilpotent, solvable, simple, semi-simple, compact/non-compact, abelian Lie groups and the direct and semi-direct product between them, as well as flag manifolds, were conducted (see \cite{DSAyEM, DSAy, DS, DS1, labsanmartin} and references therein) providing some insight on the dynamical behaviour of such systems.

One of the key tools in analyzing the dynamical properties of such systems is the concept of control sets in both topological and/or algebraic sense, which are maximal regions of the state space where approximate controllability holds. Within these regions, the system can be steered arbitrarily close to any point by selecting suitable control functions. In the Euclidean setting, control sets are often straightforward to describe, but for Lie groups, their structure is much richer and closely tied to the algebraic or geometric properties of the group. It should be noted that determining controllability property, characterizing eventual topological properties of control sets of all LCSs becomes highly non-trivial job. For example, even in the case of low-dimensional groups, the properties of control sets for such dynamics on Lie groups and homogeneous spaces might differ significantly (see \cite{VAD1, VAD2, PJMD}). In particular, for LCSs on nilpotent Lie groups, the properties of control sets is strongly influenced by the eigenvalues of a derivation associated with the drift vector field. Consequently, understanding the dynamics on this type of Lie groups is essential for gaining insights into the dynamics on more general Lie groups. In this manner, the Heisenberg group is a particularly interesting setting in which to explore these ideas. As a three-dimensional, nilpotent but non-abelian Lie group, it combines a simple structure with a highly non-trivial geometry. Given its central role in areas such as harmonic analysis, quantum mechanics and sub-Riemannian geometry, it is an ideal testing ground for understanding how group structure influences controllability. Indeed, for the regular case (where the related derivation is invertible), one of the previous works \cite{DS1} assures the existence of control sets with nonempty interior under the Lie Algebra Rank Condition (LARC). However, a critical gap remains: the LARC, while powerful, is not sufficient to guarantee the existence of such control sets in all situations. As the examples in this paper will demonstrate, the singular case (where the related derivation is non-invertible) presents a fundamentally different and more complex landscape, where controllability can either be fully achieved or entirely disintegrate. It is this unexplored and challenging singular problem that our study addresses.

In this paper, we provide a complete classification of control sets for singular LCSs on the Heisenberg group, where the associated derivation w.r.t the drift vector field is non-invertible. By leveraging the group's automorphisms to conjugate the system into simplified normal forms, we conduct a meticulous analysis based on the fundamental invariants of the derivation: the trace ($\tr A$) and determinant ($\det A$) of its $2 \times 2$ submatrix $A$. Our results reveal a rich variety of behaviors: in the case $\det A = \tr A = 0$, we uncover a distinct dichotomy: the system either exhibits global controllability (with $\mathbb{H}$ as the unique control set) or a complete breakdown of controllability, resulting in a continuum of one-point control sets (equilibria), dictated by the ad-rank condition. Furthermore, for the case $\det A \neq 0$ and $\tr A = 0$, the characterization depends on the spectrum of $A$ and the parameter $\alpha$.  When the ad-rank condition holds ($\alpha \neq 0$), the only control set is the cylinder $\mathcal{C}_{\mathbb{R}^2} \times \mathbb{R}$. If the ad-rank condition fails ($\alpha = 0$), the outcome is determined by the eigenvalues: for pure imaginary eigenvalues, the cylinder remains the control set, while for real eigenvalues, controllability collapses into a line of one-point control sets. Finally, for the case $\det A = 0$ and $\tr A \neq 0$, we demonstrate that the unique control set is the preimage $\pi^{-1}(\mathcal{C}_{\mathbb{R}^2}^A)$ of a control set from an associated affine system on $\mathbb{R}^2$. Through a detailed case-by-case study, this work uncovers the intricate and sometimes unpredictable controllability patterns that arise in the singular setting, provides a comprehensive analytical framework tailored to the Heisenberg group, and lays a concrete foundation for extending the study of singular LCSs to higher-dimensional nilpotent Lie groups.

\section{Preliminaries}
This section introduces some foundational concepts from dynamical systems to help readers understand the rest of the paper. First, we introduce control-affine systems on smooth manifolds and present several key results that will be referenced later.
\subsection{Control-affine systems and controllability}
Consider a smooth $C^\infty$ manifold  $M$ of finite dimension and the Euclidean space  $\mathbb{R}^{m}$. Let  $\Omega \subset \mathbb{R}^{m}$ be a compact, convex subset whose interior contains the origin. A control-affine system on  $M$ is then given by the family of ordinary differential equations
\begin{equation*}
\Sigma _{M}:\quad \dot{x}(\tau)=f_{0}(x(\tau))+\sum_{j=1}^{m}u_{j}(\tau)f_{j}(x(\tau)),%
\quad \mathbf{u}\in \mathcal{U},
\end{equation*}%
where $f_{0},f_{1},\ldots ,f_{m}$ are smooth vector fields defined on $M$
and the control parameter $\mathbf{u}=\left( u_{1},\ldots ,u_{m}\right) $ belongs to the set $\mathcal{U}$ of the piecewise constant functions such that $\mathbf{u}(t)\in \Omega $.

Given an initial state $x\in M$ and a control $\mathbf{u}\in\mathcal{U}$, the system $\Sigma_M$ admits a unique solution $\tau\mapsto \varphi (\tau,x,\mathbf{u})
$ which is an absolutely continuous curve on $M$ satisfying $ \varphi(0,x,\mathbf{u})=x$
 and whose derivative almost everywhere agrees with the right‑hand side of $\Sigma_M$. Associated to $\Sigma _{M}$ we have
for a given $x\in M$, the positive/negative orbits at $x$ as follows:
\begin{equation*}
\mathcal{O}^{\pm}(x) =\{\varphi (\pm\tau,x,\mathbf{u}):\tau\geq 0,\mathbf{u}\in \mathcal{U}\}. %\\
%\mathcal{O}^{-}(x) &=&\{\varphi (-\tau,x,\omega):\tau\geq 0,u\in \mathcal{U}\}\text{.}
\end{equation*}

\begin{definition}
A control-affine system $\Sigma_M$ is said to satisfy the Lie Algebra Rank Condition (LARC) if
$\mathcal{L}(x) = T_xM $ \text{for all } $x\in M,$
where $\mathcal{L}$ denotes the smallest Lie subalgebra of the space of smooth vector fields on $M$ that contains the vector fields $f_0, f_1, \dots, f_m$. The system $\Sigma_M$ is said to be controllable if, for every $x\in M$, the positive orbit $\mathcal{O}^+(x)$ of $x$ coincides with the entire manifold; that is, $\mathcal{O}^+(x) = M $ \text{for all } $x\in M .$ 
\end{definition}

Since achieving global controllability is generally difficult, it is natural to investigate the existence of maximal subsets of the state space in which controllability holds. In control theory, these subsets are known as control sets, which are defined below.
\begin{definition}
\label{controlset}
A nonempty set $\mathcal{C}\subset M$ is a control set of $\Sigma _{M}$ if it is
maximal, w.r.t. set inclusion, with the following properties:
\begin{enumerate}
\item[(1)] $\forall x\in \mathcal{C}$, there exists a control $\mathbf{u}\in \mathcal{U}$
such that $\varphi \left( \mathbb{R}^{+},x,\mathbf{u}\right) \subset \mathcal{C}$;
\item[(2)]  It holds that $\mathcal{C}\subset \mathrm{cl}~\mathcal{O}^{+}(x)$ for
all $x\in \mathcal{C}$.
\end{enumerate}
\end{definition}
As shown in \cite[Proposition 3.2.4]{FCWK}, any subset $\mathcal{C} \subset M$ with nonempty interior that is maximal with respect to property (2) above constitutes a control set. These sets provide a natural approach to analyzing essential dynamical features of the system, such as equilibrium points, recurrent behavior, periodic trajectories, and bounded orbits. Moreover, if the system satisfies the LARC, then exact controllability holds in the interior of any control set (see \cite[Theorem 3.1.5]{FCWK}).   In fact that under the LARC, a precise relationship holds between the local structure of reachable sets and control sets. Specifically, a point $x \in M$ belongs to the interior of its positive orbit, $\operatorname{int} \mathcal{O}^+(x)$, if and only if it lies in the intersection $\operatorname{int} \mathcal{O}^+(x) \cap \operatorname{int} \mathcal{O}^-(x)$, and this is further equivalent to $x$ being in the interior of a control set $\mathcal{C} $. In this case, the control set $\mathcal{C} $ can be characterized as the intersection $\overline{\mathcal{O}^+(x)} \cap \mathcal{O}^-(x)$, capturing the maximal region around $x$ where exact controllability holds.
\par
Now, let us turn to the notion of conjugation between control-affine systems, an important tool in developing our main results. A conjugation simplifies the dynamical analysis by enabling coordinate changes on the state space while preserving the essential control-theoretic structures. Let $\Sigma_M$ and $\Sigma_N$ denote control-affine systems on the smooth manifolds $M$ and $N$, respectively, with associated families of vector fields $\mathbf{f} = (f_0, f_1, \ldots, f_m)$ and $\mathbf{g} = (g_0, g_1, \ldots, g_m)$. We now state the following

\begin{definition}
Let $\psi: M \to N$ be a smooth map. A vector field $X$ on $M$ and a vector field $Y$ on $N$ are $\psi$-conjugated (or $\psi$-related) if
$$d\psi \circ X = Y \circ \psi.$$
In particular, two control-affine systems $\Sigma_M$ and $\Sigma_N$ with vector fields $\mathbf{f}$ on $M$ and $\mathbf{g}$ on $N$ are said to be $\psi$-conjugated if
$$d\psi \circ f_j = g_j \circ \psi, \quad \text{for each } j=0, \ldots, m.$$
If $\psi$ is a diffeomorphism, the systems $\Sigma_M$ and $\Sigma_N$ are called equivalent.
\end{definition}
Several important properties of equivalent systems, including controllability, topological characteristics of positive and negative orbits, and control set structure, are preserved under conjugation. The following result illustrates the relationship between the control sets of conjugated systems \cite[Proposition 2.4]{DDK}.

\begin{proposition}\label{CCC}
Let $\Sigma_M$ and $\Sigma_N$ be $\psi$-conjugated systems satisfying the LARC. Then the the followings are satisfied:
\begin{itemize}
    \item[(1)] If $\mathcal{C}_M$ is a control set of $\Sigma_M$, there exists a control set $\mathcal{C}_N$ of $\Sigma_N$ such that $\psi\left(\mathcal{C}_M\right) \subset \mathcal{C}_N$
    \item[(2)] If for some $y_0 \in \operatorname{int} \mathcal{C}_N$ it holds that $\psi^{-1}\left(y_0\right) \subset \operatorname{int} \mathcal{C}_M$, then $\mathcal{C}_M=\psi^{-1}\left(\mathcal{C}_N\right)$.    
\end{itemize}
\end{proposition}

\section{Linear control systems and Heisenberg group}
This section introduces the concept of a linear control system (abbrev. LCS) on a general Lie group and presents some of its fundamental properties. Then, particular attention will be given to how such systems are defined on the Heisenberg group. This group will serve as the main setting for the subsequent analyses.

\begin{definition}
A vector field $\mathcal{X}$ on a connected Lie group $G$ is linear if its flow $\left\{\varphi_\tau\right\}_{\tau \in \mathbb{R}}$ is a 1-parameter subgroup of $\operatorname{Aut}(G)$, the group of all automorphisms of $G$.    
\end{definition}
\begin{remark}\label{linvec}
On any connected Lie group $G$, a linear vector field is always complete and naturally determines a derivation $\mathcal{D} = -\operatorname{ad}(\mathcal{X})$ on the Lie algebra $\mathfrak{g}$, satisfying the Leibniz rule: $\mathcal{D}[X, Y] = [\mathcal{D}X, Y] + [X, \mathcal{D}Y]$ for all $X, Y \in \mathfrak{g}$. Although every linear vector field induces a derivation in this way, the converse holds only when the group is simply connected. In particular, if $G$ is a connected and simply connected nilpotent Lie group, the exponential map $\exp: \mathfrak{g} \to G$ is a global diffeomorphism. This property allows for an explicit reconstruction of the linear vector field associated with a given derivation $\mathcal{D}$. Specifically, the flow $\varphi_\tau$ of the corresponding vector field satisfies $(d\varphi_\tau)_e = e^{\tau\mathcal{D}}$, and for any $Y \in \mathfrak{g}$, one has $\varphi_\tau(\exp Y) = \exp(e^{\tau\mathcal{D}} Y)$. Since the exponential map is invertible in this setting, the drift vector field $\mathcal{X}$ can be explicitly computed from the derivation using the logarithmic map $\log(p) = Y$, where $p \in G$. This relationship between algebraic and geometric structures is important in analyzing linear control systems on nilpotent Lie groups.
\end{remark}

\begin{definition}
A linear control system on $G$ is determined by the family of ODEs
\begin{equation*}
\Sigma _{G}:\quad \dot{x}(\tau)=\mathcal{X}(x(\tau))+\sum_{i=1}^m u_i(\tau) Z_i(x(\tau)),
\quad \mathbf{u}\in \mathcal{U},
\end{equation*}%
where the drift $\mathcal{X}$ is a linear vector field, $Z_i$'s are left-invariant vector fields, and $\mathbf{u}=\left(u_1, \ldots, u_m\right) \in \mathcal{U}$ are control functions as defined previously.
\end{definition}
Thanks to the inherent symmetry of Lie groups, we can verify the LARC for LCSs at the identity element of the group. An LCS $\Sigma_G$ satisfies the LARC if the Lie algebra $\mathfrak{g}$ is the smallest $\mathcal{D}$-invariant subalgebra containing the control vectors $\{Z_1, \ldots, Z_m\}$, where $\mathcal{D}$ is the derivation associated with the drift vector field. A stronger algebraic criterion is the ad-rank condition, which is met when 
$$
\mathfrak{g} = \operatorname{span}\left\{ \mathcal{D}^k Z_j :\ 0 \leq k < \dim \mathfrak{g},\ 1 \leq j \leq m \right\}.
$$
In other words, the ad-rank condition requires that the smallest $\mathcal{D}$-invariant subspace generated by  the vectors $\{Y_1, \ldots, Y_m\}$ spans the entire Lie algebra $\mathfrak{g}$. The main distinction between LARC and the ad-rank condition is that LARC involves closure under Lie brackets (i.e., forming a subalgebra), whereas the ad-rank condition considers only the action of the derivation $\mathcal{D}$, without requiring bracket closure. Moreover, if the ad-rank condition holds, the system is locally controllable at the identity. 

\subsection{LCSs on Heisenberg Group}
This section presents the structure of the Heisenberg group and its Lie algebra. It describes the forms of invariant and linear vector fields and introduces the expression of linear control systems on the group. Additionally, it outlines sufficient conditions for controllability, with a focus on the Ad-rank condition and the LARC of the obtained systems.
\par
The {\it Heisenberg group} $\HB$ is by definition $\HB:=(\R^2\times\R, *)$ with 
$$(\mathbf{v}_1, z_1)*(\mathbf{v}_2, z_2):=\left(\mathbf{v}_1+\mathbf{v}_2, z_1+z_2+\frac{1}{2}\omega(\mathbf{v}_1, \mathbf{v}_2)\right),$$
where $\omega(\mathbf{v}_1, \mathbf{v}_2):=\det(\mathbf{v}_1|\mathbf{v}_2)$ is the determinant of the matrix having $\mathbf{v}_1, \mathbf{v}_2$ as columns, that is, the unique (up to nonzero constant multiplication) nondegenerated, skew symmetric bilinear form on $\R^2$.  It is a standard fact that the $\HB$ is in fact a Lie group and, up to an isomorphism, is the unique three-dimensional nilpotent, nonabelian simply connected Lie group.

The Lie algebra $\fh$ of $\HB$ is given by $\fh:=(\R^2\times \R, [\cdot, \cdot])$ with
$$[(\zeta_1, \alpha_1), (\zeta_2, \alpha_2)]:=\left(\mathbf{0}, \omega(\zeta_1, \zeta_2)\right).$$
The next result presents the structure of the Lie algebra derivations that are used to determine linear vector fields, as well as the group automorphisms that govern transitions between different control systems on the group.
\begin{proposition}\label{deraut}
Assume that $GL(2)$ denotes the Lie group of $2 \times 2$ invertible matrices with Lie algebra $\mathfrak{g l}(2)$. Then the explicit form of a derivation $\mathcal{D}$ for $\mathfrak{h}$ and an automorphism for $\mathbb{H}$ both in matrix form w.r.t. the standard basis is as follows:
$$
\mathcal{D}=\left(\begin{array}{cc}
A & \mathbf{0} \\
\eta^{\top} & \operatorname{tr} A
\end{array}\right) \in \operatorname{Der}(\mathfrak{h}) \quad \text { and } \quad \mathcal{P}=\left(\begin{array}{cc}
P & \mathbf{0} \\
\xi^{\top} &\operatorname{det} P
\end{array}\right) \in \operatorname{Aut}(\mathbb{H}) .
$$
where $P \in \mathrm{GL}(2), A \in \mathfrak{g} \mathfrak{l}(2)$, and $\eta, \xi \in \mathbb{R}^2$.
\end{proposition}
Now, the elements of the system of interest can be introduced in a structured manner. For a left-invariant vector field $Z$, it holds that if
$$Z=(\zeta, \alpha)\in\fh\hspace{.5cm}\mbox{ and }\hspace{.5cm} \hspace{.5cm} g=(\mathbf{v}, z)\in\HB\hspace{.5cm}\mbox{ then }\hspace{.5cm}Z(g)=\left(\zeta, \alpha+\frac{1}{2}\omega(\mathbf{v}, \zeta)\right).$$
By the Remark \ref{linvec} the action of $\DC$ at an element $(\mathbf{v}, z)\in\mathbb{H}$ is given by matrix multiplication and is written, in coordinates, as 
$$\DC(\mathbf{v}, z)=\left(\begin{array}{cc}
A & {\bf 0}\\ \eta^{\top} & \tr A
\end{array}\right)\left(\begin{array}{c}
\mathbf{v}\\ z
\end{array}\right)=(A\mathbf{v}, \omega(v, \theta\eta)),$$
where $\theta$ is the counter-clockwise rotation by $\pi/2$. 
%{\color{red} usefull?
%It follows that the flow $\varphi_t$ induced by $\mathcal{X}$ takes the form
%$$
%\varphi_t(\mathbf{v}, z)=\left(e^{t A} \mathbf{v},\left\langle\mathrm{e}^{\tau \cdot \operatorname{tr} A} \boldsymbol{\Lambda}_t^{\left(A-\operatorname{tr} A \cdot I_2\right)} \eta, \mathbf{v}\right\rangle+z \mathrm{e}^{t \cdot \operatorname{tr} A}\right)
%$$
%where $I_2$ stands for the $2 \times 2$ identity matrix and $\Lambda^B: \mathbb{R} \times \mathbb{R}^2$ is an operator defined by $\Lambda_t^B \eta=$ $\int_0^t e^{s B^{\top}} \eta d s$ for all $B \in \mathfrak{g} \mathfrak{l}(2)$.}
Consequently, a (one-input) LCS on $\mathbb{H}$ is, in coordinates, given as
 \begin{equation}\label{linearonH}
        \left(\Sigma_{\HB}\right):\quad \left\{\begin{array}{l}
     \dot{\mathbf{v}}= A\mathbf{v}+u\zeta\\
     \dot{z}= z\tr A +u\alpha+\omega\left(\mathbf{v}, \theta\eta+u\frac{1}{2}\zeta\right)
\end{array}\right.
    \end{equation} 
where $u\in \Omega:=[u_{*}, u^{*}]$ with $u_{*}<0<u^{*}$. Moreover, we assume that $\DC\not\equiv 0$ and $\alpha^2+|\zeta|^2\neq 0$, to avoid trivial cases. In particular, the first equation gives a linear control system on $\mathbb{R}^2$ that is conjugated to $\Sigma_{\R^2}$ through the canonical projection 
    $$\pi:\HB\rightarrow\R^2, \hspace{.5cm}(\mathbf{v}, z)\mapsto \mathbf{v}.$$
In particular, the LARC of the $\Sigma_{\HB}$ implies the Kallman-rank condition of $\Sigma_{\R^2}$, assuring the existence of a unique control set $\CC_{\R^2}$ satisfying ${\bf 0}\in\inner\CC_{\R^2}$ (see for instance \cite{DSAyAR, DSAyEM}). In particular, any control set of $\Sigma_{\HB}$ has to be contained in the preimage $\pi^{-1}(\CC_{\R^2})=\CC_{\R^2}\times\R$. These facts will help us to prove the existence and uniqueness of control sets for $\Sigma_{\HB}$ in some cases, as we will see below. 
    
The first relationship is provided in the following technical lemma.

\begin{lemma}
\label{preimage}
If $\Sigma_{\HB}$ satisfies the LARC and the fiber $\{{\bf 0}\}\times\R$ is controllable, then $\CC_{\R^2}\times\R$ it the unique control set with nonempty interior of $\Sigma_{\HB}$. 
\end{lemma}

\begin{proof}
Let us start by noticing that, the controllability in $\inner\CC_{\R^2}$ and on $\{{\bf 0}\}\times\R$, imply that $\inner\CC_{\R^2}\times\R$ is controllable and, in particular, 
$$\forall (\mathbf{v}, z)\in \CC, \hspace{.5cm}\mbox{ it holds that }\hspace{.5cm}\CC\subset\overline{\OC^+(\mathbf{v}, z)}.$$
Therefore, there exists a control set $\CC$ of $\Sigma_{\HB}$ satisfying $\CC_{\R^2}\times\R\subset\CC$. However, by Proposition \ref{CCC} and the fact that the canonical projection 
$$\pi:\HB\rightarrow\R^2, \hspace{.5cm}(\mathbf{v}, z)\mapsto \mathbf{v},$$
    conjugates the linear control systems $\Sigma_{\HB}$ and $\Sigma_{\R^2}$, we get that
    $$\pi^{-1}({\bf 0})=\{\bf 0\}\times\R\subset\inner\CC\hspace{.5cm}\implies\hspace{.5cm} \CC=\pi^{-1}(\CC_{\R^2})=\CC_{\R^2}\times\R,$$
concluding the proof.
\end{proof}

%\PC^{-1}=\left(\begin{array}{cc}
%	P^{-1} & \mathbf{0}\\ -(\det P^{-1})(P^{-1})^{\top}\eta^{\top} & \det P^{-1}
%	\end{array}\right) \quad \text{with} \quad 

\begin{remark}
\label{remark}
Another important conjugation we can make of your initial LCS comes from the use of automorphisms. In fact, since $\HB$ is connected and simply connected, \cite[Proposition 7]{PJMD} assures that, for any automorphism $\PC$,  a linear control system with associated derivation $\DC$ and a left-invariant vector field $Z$ is equivalent to the linear control system whose derivation is given by $\PC\DC\PC^{-1}$ and the left-invariant vector field is $\PC Z$. 

Hence, by the Proposition \ref{deraut}, taking $\PC=\left(\begin{array}{cc}
P & \mathbf{0} \\
\xi^{\top} &\operatorname{det} P
\end{array}\right)\in \operatorname{Aut}(\mathbb{H})$ we have
$$\PC\DC\PC^{-1}=\left(\begin{array}{cc}
	PAP^{-1} & \mathbf{0}\\ \widehat{\eta}^{\top} & \tr A
	\end{array}\right),\hspace{.5cm}\widehat{\eta}=(P^{-1})^{\top}\left((A-\tr A\cdot I_2)^{\top}\xi+\det P\eta\right)$$
and 
$$\PC Z=(P\zeta, \omega(\zeta, \theta\eta)+\alpha\det P),$$
 where in the previous $I_2$ stands or the identity map of $\R^2$.  As can be seen that the top-left block of  $\PC\DC\PC^{-1}$  is  $PAP^{-1}$ and, since $P \in \mathrm{GL}(2)$ is arbitrary, we can choose $P$ such that $PAP^{-1}$ is in Jordan canonical form. This simplifies  $A$ while preserving its spectral properties (i.e., $\det A$, $\tr A$ and its eigenvalues).  Moreover, the vector $\eta\in\R^2$ adjusts via
$\widehat{\eta}=(P^{-1})^{\top}\left((A-\tr A\cdot I_2)^{\top}\xi+\det P\eta\right) 
$. Hence, if $\det A\neq 0$ we can conjugated our initial system to a system where the associated derivation has $\widehat{\eta}=0$. On the other hand, one can choose an automorphism $\PC$ that changes our the first component $\zeta$ of $Z$ or even has $\alpha=0$. 

Therefore, choosing appropriated automorphisms allows us to make strategic choices that greatly simplifies the calculations involved as we will see ahead.
\end{remark}

We finish this section by providing a simpler algebraic way to assure the LARC and the ad-rank condition.

\begin{lemma}
Assume that the $\Sigma$ be an LCS defined in the system (\ref{linearonH}). Then the followings are satisfied:
\begin{itemize}
    \item[(1)] $\Sigma_{\HB}$ satisfies ad-rank condition if and only if $\omega(A\zeta, \zeta)\left(\alpha\det A+\omega(A\zeta, \theta\eta)\right)\neq 0.$
    \item[(2)] $ \Sigma_{\HB}$  satisfies the LARC  if and only if $\omega(A\zeta, \zeta)\neq 0.$
\end{itemize}
\end{lemma}

\begin{proof} (1) Let us start by noticing that if $Z=(\zeta, \alpha)$, then 
$$\DC Z=(A\zeta, \alpha\tr A+\omega(\zeta, \theta\eta))\hspace{.5cm}\mbox{ and }\hspace{.5cm} [Z, \DC Z]=(\mathbf{0}, \omega(A\zeta, \zeta)).$$
On the other hand, by Cayley-Hamilton, 
$$A^2\zeta=-\det A\zeta+\tr A\cdot A\zeta,$$
and hence
$$\DC^2 Z=(A^2\zeta, \omega(A\zeta, \theta\eta)+\tr A\cdot\omega(\zeta, \theta\eta)+\alpha(\tr A)^2)=-\det A \cdot Z+\tr A\cdot\DC Z+(\mathbf{0}, \alpha\det A+\omega(A\zeta, \theta\eta)).$$
Since $\omega(A\zeta, \zeta)\neq 0$ if and only if $\{A\zeta, \zeta\}$ is linearly independent, we conclude that $\Sigma_{\HB}$ satisfies the ad-rank condition if and only if  $\{Z, \DC Z, \DC^2 Z\}$ is linearly independent if and only if
$$\omega(A\zeta, \zeta)(\alpha\det A+\omega(A\zeta, \theta\eta))\neq 0.$$

(2) If $\omega(A\zeta, \zeta)=0$ then $\det A=0$ and  $A\zeta=\tr A\cdot\zeta$ implying, by the previous item, that
$$\DC^2 Z=\tr A\cdot\DC Z+\tr A\cdot (\mathbf{0}, \omega(\zeta, \theta\eta)).$$
On the other hand,
$$\DC Z=(\tr A\cdot\zeta, \alpha\tr A+\omega(\zeta, \theta\eta))=\tr A\cdot(\zeta, \alpha)+(\mathbf{0}, \omega(\zeta, \theta\eta))=\tr A\cdot Z+(\mathbf{0}, \omega(\zeta, \theta\eta)),$$
implying that 
$$\DC^2 Z=2\tr A\cdot\DC Z-(\tr A)^2\cdot Z.$$
Since, $[Z, \DC Z]=(\mathbf{0}, \omega(A\zeta, \zeta))=(\mathbf{0}, 0)$ we conclude that
$$\omega(A\zeta, \zeta)=(\mathbf{0}, 0)\hspace{.5cm}\implies\hspace{.5cm}\Sigma_{\HB} \mbox{ cannot satisfy the LARC}.$$

Reciprocally, it holds that  
$$a_1Z+a_2+\DC Z+a_2[Z, \DC Z]=(\mathbf{0},0)\hspace{.5cm}\iff\hspace{.5cm} \left\{\begin{array}{c}
	a_1\zeta+a_2 A\zeta=\mathbf{0}\\
	a_1\alpha+a_2(\alpha\tr A+\omega(\zeta, \theta\eta))+a_3\omega (A\zeta, \zeta)=0
	\end{array}\right.$$
In particular, $\omega(A\zeta, \zeta)\neq 0$ implies that $a_1=a_2=0$ on the first equation, and hence, 
$$a_3\omega(A\zeta, \zeta)=0\hspace{.5cm}\implies\hspace{.5cm} a_3=0,$$
showing that
$$\omega(A\zeta, \zeta)\neq 0\hspace{.5cm}\implies\hspace{.5cm}\mathrm{span}\{Z, \DC Z, [Z, \DC Z]\}=\fh,$$
concluding $\Sigma_{\HB}$ satisfies the LARC.
\end{proof}

%%%%%%%%%%%%%%%%%%%%% we may delete the following automorphism lemma and other things except for fiber controllability since i wrote it above %%%%%%%%%%%%%%%%%%%%%%%%%%%%%%%%%%%%%%%%%%%%

%%%%%%%%%%%%%%%%%%%%%%%%%%%%%%%%%%%%%%
%{\color{red}
%	\begin{lemma}
%		The control sets of $\Sigma$ are contained in the pre-image of the control sets of $\Sigma_{\R^2}$.
%	\end{lemma}

%The {\it Heisenberg group} $\HB$ is by definition $\HB:=(\R^2\times\R, *)$ with 
%$$(v_1, z_1)*(v_2, z_2):=\left(v_1+v_2, z_1+z_2-\frac{1}{2}\omega(v_1, %v_2)\right),$$
%where $\omega(v_1, v_2):=\det(v_1 v_2)$ is the determinant of the %matrix having $v_1, v_2$ as columns. {\color{red}Note that this is independent of the basis of $R^2$ choosen.}}

\section{The control sets of $\HB$}
In this section, control sets with nonempty interior for LCSs on the Heisenberg group are analyzed. The discussion is organized according to the different cases arising from whether the determinant and trace of the matrix 
$A$ vanish or not. Since the LARC provides the minimal requirement for the existence of control sets with nonempty interior, it will be assumed throughout the remainder of the analysis.

\subsection{The case $\det A=\tr A=0$}
%$$A\equiv 0\;\;\;\mbox{ or }\;\;\;A=\left(\begin{array}{cc}
%0 & 1\\ 0 & 0
%\end{array}\right).
%$$
From Proposition \ref{deraut} and Remark \ref{remark}, the automorphism
$$\mathcal{P}=\left(\begin{array}{cc}
   \frac{1}{2}I_2  & {\bf 0} \\
    -\frac{\alpha}{4|\zeta|^2} & \frac{1}{4}
\end{array}\right),
$$
is an automorphism of $\HB$ that conjugates the linear control system $\Sigma_{\HB}$ to the linear control system of the form  
$$\left\{\begin{array}{l}
     \dot{\mathbf{v}}= A\mathbf{v}+u\zeta\\
     \dot{z}=\frac{1}{2}\omega\left(\mathbf{v}, \theta\eta+u\zeta\right)
     \end{array}\right.
     $$

Hence, 
$\Sigma_{\HB} $ verifies the ad-rank condition if and only if $ \omega(A\zeta, \theta\eta)\neq 0.$ This immediately leads to the following theorem:
\begin{theorem}
\label{bothzeros}
	It holds:
	\begin{itemize}
		\item[(1)] If $\omega(A\zeta, \theta\eta)\neq 0$ then $\HB$ is the only control set of $\Sigma_{\HB}$,
		\item[(2)]If $\omega(A\zeta, \theta\eta)=0$ then the plane $\ker\DC$ is a continuum of one-point control sets of $\Sigma_{\HB}$. 
	\end{itemize} 
\end{theorem}

\begin{proof} Since, by the LARC,  $\omega(A\zeta, \zeta)\neq 0$, the set $\{(A\zeta, 0), (\zeta, 0), (0, \frac{1}{2}\omega(A\zeta, \zeta))\}$ forms a basis of $\fh$. In such a basis, the linear control system $\Sigma_{\HB}$ is written as
$$
\hspace{-2cm}(\Sigma_{\HB})\hspace{2cm}\left \{
\begin{array}
[c]{l}%
\dot{x}=y\\
\dot{y}= u\\
\dot{z}=ax+by+xu
\end{array}
\right.,\hspace{.5cm} \mbox{ where }\hspace{.5cm}a:=\frac{\omega(A\zeta, \theta\eta)}{\omega(A\zeta, \zeta)}\hspace{.5cm}\mbox{ and }\hspace{.5cm}b:=\frac{\omega(\zeta, \theta\eta)}{\omega(A\zeta, \zeta)}.
$$

The map, 
$$f:\HB\rightarrow\R^3, \hspace{1cm}f(x, y, z)=\left(x, y, z-b x-\frac{1}{3}xy\right),$$
conjugates the system $\Sigma_{\HB}$ to the system\footnote{The change from $\HB$ to $\R^3$ in the subscripts is to emphasize that $\Sigma_{\R^3}$ is not a linear control system.}
$$\hspace{-2cm}(\Sigma_{\R^3})\hspace{2cm}\left \{
\begin{array}
[c]{l}%
\dot{x}=y\\
\dot{y}= u\\
\dot{z}=ax+\frac{1}{3}(2x u-y^2)
\end{array}
\right.,$$
whose solutions for constant control functions are
$$\phi(t, (\mathbf{v}_0, z_0), u)=\left(x_0+y_0t+\frac{ut^2}{2}, y_0+tu, z_0+a\left(x_0t+y_0\frac{t^2}{2}+\frac{ut^3}{6}\right)+\frac{1}{3}(2x_0u-y_0^2)t\right).$$

(1) In this case, $a\neq 0$ and the linear control system $\Sigma_{\R^2}$, induced by the first two components of $\Sigma_{\R^3}$ is controllable (see \cite[Theorem 3.1]{DSAyAR}). In particular, for any $\mathbf{v}_1, \mathbf{v}_2\in\R^2$ there exists $\tau>0$ and $ u\in\UC$ such that 
$$\phi(t, \{\mathbf{v}_1\}\times\R,  u) = \{{\bf v_2}\}\times\R.$$

Let us use the previous to show that $\{\mathbf{0}\}\times\R$ is controllable. Let then $(\mathbf{0}, z_1), (\mathbf{0}, z_2)$ and assume $z_1<z_2$ since the other case is analogous. Let $\mathbf{v}=(x, 0)\in\R^2$ with $ax>0$ and consider controls $u_0, u_1, u_2\in\UC$, times $t_0, t_1, t_2>0$ real numbers $z_0, \bar{z_0}, \bar{z_1}, \bar{z_2}\in\R$ such that $\bar{z}_1<z_0$, $\bar{z_2}<\bar{z_0}$,
$$\phi(t_0, (\mathbf{v}, z_0), u_0)=(-\mathbf{v}, \bar{z}_0), \hspace{.5cm}\phi(t_1, (\mathbf{0}, z_1)), u_1)=(\mathbf{v}, \bar{z}_1)\hspace{.5cm}\mbox{ and }\hspace{.5cm}\phi(t_2, (-\mathbf{v}, \bar{z}_2), u_2)=(\mathbf{0}, z_2).$$
A trajectory connecting $(\mathbf{0}, z_1)$ to $(\mathbf{0}, z_2)$ is constructed as follows:

\begin{itemize}
    \item[(a)] With control $u_1$ and time $t_1>0$, connect $(\mathbf{0}, z_1)$ to the point $(\mathbf{v}, \bar{z}_1)$;

\item[(b)] With control $u\equiv 0$ and time $t=\frac{z_0-\bar{z_1}}{ax}>0$ connect $(\mathbf{v}, \bar{z}_1)$ to the point
$$\phi(t, (\mathbf{v}, \bar{z}_1), 0)=(\mathbf{v}, \bar{z}_1+axt)=(\mathbf{v}, z_0);$$

\item[(c)] With control $u_0$ and time $t_0>0$ go from $(\mathbf{v}, z_0)$ to $(-\mathbf{v}, \bar{z}_0)$;

    \item[(d)] Now, with control $u\equiv 0$ and time $t'=\frac{\bar{z_0}-z_2}{ax}>0$ connect $(-\mathbf{v}, \bar{z}_0)$ to the point
$$\phi(t', (-\mathbf{v}, \bar{z}_0), 0)=(-\mathbf{v}, \bar{z}_0-axt')=(-\mathbf{v}, \bar{z}_2);$$

\item[(e)] With control $u_2$ and time $t_2>0$ go from $(-\mathbf{v}, \bar{z}_2)$ to $(\mathbf{0}, z_2)$, concluding the proof.
\end{itemize}

(2) By our basis choice, the plane $\ker\DC$ coincides with the plane $y=0$. In particular, the fact that the conjugation $f$ fixes this plane implies that we can work with the system $\Sigma_{\R^3}$. Then, the hypothesis $\omega(A\zeta, \theta\eta)=0$ implies $a=0$ and the solutions of $\Sigma_{\R^3}$ are reduced to
$$\phi(t, (\mathbf{v}_0, z_0), u)=\left(x_0+y_0t+\frac{ut^2}{2}, y_0+tu, z_0+\frac{1}{3}(2x_0u-y_0^2)t\right).$$

 Since the points in $\ker\DC$ are equilibria of the system, any of them is contained in a control set. Therefore, we only have to show that these control sets are in fact singletons and they are the only control sets of  $\Sigma_{\R^3}$. In order to do that, let us define the function 
$$F:\R^3\rightarrow\R, \hspace{1cm}F(x, y, z):=3z\sigma+y(y^2-2x\sigma),$$
where $\sigma< u$ for all $u\in\Omega$. It is straightforward to see that any $F^{-1}(c)$ is a smooth deformation of a plane for any $c\in\R$. In particular, it divides $\R^3$ into the disjoint half-spaces
$$F^+_c:=\{(\mathbf{v}, z)\in\R^3; F(\mathbf{v}, z)>c\}\hspace{.5cm}\mbox{ and }\hspace{.5cm}F^-_c:=\{(\mathbf{v}, z)\in\R^3; F(\mathbf{v}, z)<c\}.$$

Now, for any $u\in\Omega$ and $\mathbf{v}_0=(x_0, y_0, z_0)$, it holds that: If $u=0$, then 
$$F(\phi(t, (\mathbf{v}_0, z_0), 0))=F\left(x_0+ty_0, y_0, z_0-\frac{1}{3}y_0^2t\right)=3\left(z_0-\frac{1}{3}y_0^2t\right)\sigma+y_0(y_0^2-2(x_0+ty_0)\sigma)$$
$$=F(\mathbf{v}_0, z_0)-3ty_0^2\sigma\geq F(\mathbf{v}_0, z_0)
,\hspace{.5cm}\mbox{ since }\hspace{.5cm} \sigma<0,$$
where the equality holds if and only if $y_0=0$. On the other hand, if $u\neq 0$, then  $$g_u(t):=F(\phi(t, (\mathbf{v}_0, z_0), u))=3\sigma\left(z_0+\frac{1}{3}(2x_0u-y_0^2)t\right)+(y_0+ut)\left((y_0+ut)^2-2\left(x_0+y_0t+\frac{ut^2}{2}\right)\sigma\right)$$
$$=(y_0+ut)^3+\sigma\left(3z_0-2x_0y_0-3y_0^2t-3y_0ut^2-u^2t^3\right)=(y_0+ut)^3+\sigma\left(3z_0-2x_0y_0+\frac{1}{u}y_0^3-\frac{1}{u}(y_0+ut)^3\right)$$
$$=\frac{u-\sigma}{u}(y_0+ut)^3+\sigma\left(3z_0-2x_0y_0+\frac{1}{u}y_0^3\right)=F(\mathbf{v}_0, z_0)+\frac{u-\sigma}{u}\Bigl[(y_0+ut)^3-y_0^3\Bigr].$$
   Derivation, gives us that 
   $$g_u'(t)=3(u-\sigma)(y_0+ut)^2\geq 0, $$
   with equality if and only if $y_0+ut=0$. Therefore, 
 $$y_0^2+u^2\neq 0\hspace{.5cm}\implies\hspace{.5cm}\forall t>0, \hspace{.5cm}F(\phi(t, \mathbf{v}_0, u))> F(\mathbf{v}_0, z_0).$$
As a consequence, 
$$\overline{\mathcal{O}^{\pm}(\mathbf{v}_0, z_0)}\setminus \{(\mathbf{v}_0, z_0)\}\subset F^{\pm}_{c}, \hspace{.5cm}\mbox{ for }\hspace{.5cm}c=F(\mathbf{v}_0, z_0),$$
and hence, 
$$\overline{\OC^+(\mathbf{v}_1, z_1)}=\overline{\OC^+(\mathbf{v}_2, z_2)}\hspace{.5cm}\iff\hspace{.5cm}(\mathbf{v}_1, z_1)=(\mathbf{v}_2, z_2).$$
By condition (2) in Definition \ref{controlset} we conclude that the control sets of $\Sigma_{\R^3}$ are singletons. On the other hand, condition (1) of the same definition implies that these singletons must be equilibria, forcing them to belong in $\ker\DC$, concluding the proof.

\end{proof}

\subsection{The case $\det A\neq 0$ and $\tr A= 0$}

Since $\Sigma_{\R^2}$ satisfies the LARC it admits a unique control set with nonempty interior $\CC_{\R^2}$. Moreover, since we are assuming that $\det A\neq 0$ and that $\omega( A\zeta, \zeta)\neq0$ we can conjugate the system and assume that $\eta=0$ (see Remark \ref{remark}). In this case, we have that 
$$\mbox{ the ad-rank condition holds for }\Sigma\iff\alpha\neq 0.$$

Moreover, under such assumptions, the $z$-component satisfies
$$\dot{z}(t)=\frac{u}{2}\omega(\mathbf{v}(t), \zeta)+u\alpha=\frac{u}{2}\omega \left(A^{-1}(\dot{\mathbf{v}}(t)-u\zeta), \zeta\right)+\alpha u=\frac{u}{2\det A}\omega \left(\dot{\mathbf{v}}(t), A\zeta\right)+\frac{u^2}{2\det A}\omega \left(A\zeta, \zeta\right)+u\alpha$$
implying that
$$z(t)-\frac{u}{2\det A}\omega \left(\mathbf{v}(t), A\zeta\right)=z_0-\frac{u}{2\det A}\omega \left(\mathbf{v}_0, A\zeta\right)+tp(u), \hspace{.5cm}\mbox{ where }\hspace{.5cm}p(u)=\frac{u^2}{2\det A}\omega \left(A\zeta, \zeta\right)+u\alpha.$$

In particular, we have that 
$$\phi(t, (\mathbf{v}(u), z_0), u)=(\mathbf{v}(u), z_0+tp(u)), \hspace{.5cm}\mbox{ where } \hspace{.5cm}\mathbf{v}(u)=-uA^{-1}\zeta,$$
is an equilibrium point of the LCS on $\mathbb{R}^2$ induced by the planar component of $\Sigma_{\HB}$.

\begin{proposition}
\label{fiber}
For any $z_1<z_2$, it holds:
\begin{itemize}
    \item[(1)] If $\alpha\neq 0$, there exists a periodic trajectory connecting $(\mathbf{0}, z_1)$ to $(\mathbf{0}, z_2)$;

    \item[(2)] If $\alpha=0$ and $(\det A)^{-1}\omega(A\zeta, \zeta)
    >0$, there exist a trajectory connecting $(\mathbf{0}, z_1)$ to $(\mathbf{0}, z_2)$ in positive-time;

    \item[(3)] If $\alpha=0$ and $(\det A)^{-1}\omega(A\zeta, \zeta)<0$, there exist a trajectory connecting $(\mathbf{0}, z_2)$ to $(\mathbf{0}, z_1)$ in positive-time
\end{itemize}

\end{proposition}

\begin{proof}
    Up to conjugation, we can assume that $\alpha\geq 0$. Moreover, let us analyze the case where $(\det A)^{-1}\omega(A\zeta, \zeta)>0$ since the other possibility is analogous. Under such assumptions, there exists $\delta>0$ such that 
    $$(0, \delta)\subset\inner\Omega\hspace{.5cm}\mbox{ and }\hspace{.5cm} p(0, \delta)\subset (0, +\infty).$$
    Take $\epsilon>0$ satisfying $3\epsilon=z_2-z_1$ and consider $t_0>0$ such that 
    $$\overline{\OC_{t_0}^+(\mathbf{0}, z_1)}\subset B_{\epsilon}(\mathbf{0}, z_1)\hspace{.5cm}\mbox{ and }\hspace{.5cm}\overline{\OC_{t_0}^-(\mathbf{0}, z_2)}\subset B_{\epsilon}(\mathbf{0}, z_2),$$
    which exists by the continuity of the solutions and the compactness of $\Omega$. Moreover, the LARC implies that $\pi(\OC_{t_0}^+(\mathbf{0}, z_1))$ and $\pi(\OC_{t_0}^-(\mathbf{0}, z_2))$ are open neighborhoods of the origin in $\R^2$, where $\pi: (\mathbf{v}, z)\in\HB\rightarrow \mathbf{v}\in \R^2$ is the canonical projection. By the exact controllability in the interior of $\CC_{\R^2}$,  there exists $u\in(0, \delta)$ and $u_1, u_2\in\UC$ such that 
    $$\phi(t_0, (\mathbf{0}, z_1), u_1)=(\mathbf{v}(u), \bar{z}_1) \hspace{.5cm}\mbox{ and }\hspace{.5cm} \phi(t_0, (\mathbf{v}(u), \bar{z}_2), u_2)= (\mathbf{0}, z_2),$$
and by our choices $\bar{z}_1<\bar{z}_2$. On the other hand, 
$$t_1=\frac{\bar{z}_2-\bar{z}_1}{p(u)}>0\hspace{.5cm}\implies\hspace{.5cm}\phi(S_1, (\mathbf{v}(u), \bar{z}_1), u)=(\mathbf{v}(u), \bar{z}_1+t_1p(u))=(\mathbf{v}(u), \bar{z}_2),$$
   and so, by concatenation, we get 
   $$\phi(t_0, \phi(t_1, \phi(t_0, (\mathbf{0}, z_1), u_1), u), u_2)=(\mathbf{0}, z_2).$$

In particular, if $\alpha>0$, we get that $p(- \delta, 0)\subset(-\infty, 0)$ and the previous process can be done with $u\in (-\delta, 0)$ to get a trajectory from $(\mathbf{0}, z_2)$ to $(\mathbf{0}, z_1)$, concluding the proof.
\end{proof}

The previous proposition shows that the fiber is controllable if $\alpha\neq 0$ and controllable in positive-time or negative-time when $\alpha=0$.

\begin{theorem}
It holds:
    \begin{itemize}

    \item[(1)] If $\Sigma_{\HB}$ satisfies the ad-rank condition, then  $\CC_{\R^2}\times\R$ is the only control set of $\Sigma_{\HB}$.

    \item[(2)] If $\Sigma_{\HB}$ does not satisfies the ad-rank condition, but 
    $A$ has a pair of pure imaginary eigenvalues, then (1) also holds;

       \item[(3)] If $\Sigma_{\HB}$ does not satisfies the ad-rank condition and $A$ has a pair of real eigenvalues, then the line $\ker\DC$ is a continuum of one-points control sets of $\Sigma_{\HB}$.

       \end{itemize}
\end{theorem}

\begin{proof}  
(1)  By the comments on the beginning of the section, up to conjugations, the ad-rank condition is equivalent to 
 $\alpha\neq 0$. Therefore, by Proposition \ref{fiber} we have that the fiber $\{\bf 0\}\times\R$ is controllable, which by Lemma \ref{preimage} gives us the result.

 (2) Let us assume that $A$ has a pair of pure imaginary eigenvalues and $\alpha=0$. In this case, we can write in some orthonormal basis, $A=\mu\theta$, where $\mu=\sqrt{|\det A|}$. On the other hand, the fact that 
 $$\frac{1}{\det A}\omega(A\zeta, \zeta)=\frac{1}{\mu}|\zeta|^2>0,$$ 
 implies by Proposition \ref{fiber} that $\{0\}\times\R$ is controllable as soon as we can construct a trajectory connecting $(\mathbf{0}, z_2)$ to $(\mathbf{0}, z_1)$ with  $z_1<z_2$. 

    In this case, the solutions of $\Sigma_{\HB}$, for constant control, can be written as
    $$\phi_1(t, (\mathbf{
    v}_0, z_0), u)=(\cos\mu t)(\mathbf{
    v_0}-\mathbf{v}(u))+(\sin\mu t)\theta(\mathbf{v}_0-\mathbf{v}(u))+\mathbf{v}(u)$$
    $$\phi_2(t, (\mathbf{v}_0, z_0), u)=-\frac{u}{2\mu}\left[(\cos\mu S)\omega(\mathbf{v}_0-\mathbf{v}(u), \zeta)-(\sin \mu t)\omega(\mathbf{v}_0-\mathbf{v}(u), \theta\zeta)\right]+t\frac{u^2}{2\mu}|\zeta|^2+z_0,$$
    where the $z$-component is obtained by the  calculations at the beginning of the section.

    Let us fix $\rho\in\Omega$ with $\rho>0$ and define $\mathbf{v}^*:=-\rho\frac{\pi}{\mu}\zeta+\mathbf{v}(\rho)$. By the controllability of the induced system on $\R^2$, there exist $u_1, u_2, u^*\in\UC$, $t_1, t_2, t^*>0$ and $\bar{z}_1, \bar{z}_2, z^*\in\R$ satisfying
    $$\phi(t_2, (\mathbf{0}, z_2), u_2)=(\mathbf{v}^*, \bar{z}_2), \hspace{.5cm}\phi(t_1, (\mathbf{v}(\rho), \bar{z}_1), u_1)=(\mathbf{0}, z_1)\hspace{.5cm} \mbox{ and }\hspace{.5cm}\phi(t^*, (\mathbf{v}^*, 0), u^*)=(\mathbf{v}(\rho), z^*).$$

    A trajectory connecting $(\mathbf{0}, z_2)$ to $(\mathbf{0}, z_1)$ is construct as follows: 

    \begin{itemize}
        \item[(a)] With control $u_2$ and time $t_2$ connect $(\mathbf{0}, z_2)$ with $(\mathbf{v}^*, \bar{z}_2)$; 
    
        \item[(b)] Using the constant control $u\equiv\rho$ and $\tau_0=\frac{\pi}{\mu}$ we have that 
        $$\phi\left(\frac{\pi}{\mu}, (\mathbf{v}^*, \bar{z}_2), \rho\right)=\left(-(\mathbf{v}^*-\mathbf{v}(\rho))+\mathbf{v}(\rho), -\frac{\rho}{2\mu}\left[-2\omega\bigl(\mathbf{v}^*-\mathbf{v}(\rho), \zeta\bigr)\right]+\rho^2\frac{\pi}{4\mu^2}|\zeta|^2+\bar{z}_2\right)$$
        $$=\left(\rho\frac{\pi}{\mu}\zeta+\mathbf{v}(\rho), -\rho^2\frac{\pi}{2\mu^2}|\zeta|^2+\rho^2\frac{\pi}{4\mu^2}|\zeta|^2+\bar{z}_2\right)=\left(\rho\frac{\pi}{\mu}\zeta+\mathbf{v}(\rho), -\rho^2\frac{\pi}{4\mu^2}|\zeta|^2+\bar{z}_2\right);$$

        \item[(c)] Since  $\mathbf{v}^*$ and $\rho\frac{\pi}{2\mu}\xi+\mathbf{v}(\rho)$ have the same norm,  there exists $t_0>0$ such that $R_{\mu t_0}\left(\rho\frac{\pi}{2\mu}\xi+\mathbf{v}(\rho)\right)=\mathbf{v}^*,$ and hence
        $$\phi\left(t_0, \phi\left(\tau_0, (\mathbf{v}^*, \bar{z}_2), \rho\right), 0\right)=\phi\left(t_0, \left(\rho\frac{\pi}{\mu}\zeta+\mathbf{v}(\rho), -\rho^2\frac{\pi}{4\mu^2}|\zeta|^2+\bar{z}_2\right), 0\right)$$
        $$=\left(R_{\mu t_0}\left(\rho\frac{\pi}{2\mu}\xi+\mathbf{v}(\rho)\right), -\rho^2\frac{\pi}{4\mu^2}|\zeta|^2+\bar{z}_2\right)=\left(\mathbf{v}^*, -\rho^2\frac{\pi}{4\mu^2}|\zeta|^2+\bar{z}_2\right);$$

        \item[(d)] Repeat items (b) and (c) $n_0$-times to obtain a trajectory connecting $(\mathbf{v}^*, \bar{z}_2)$ to the point $(\mathbf{v}^*, \hat{z}_2)$, where 
        
        $$\hat{z}_2:= -n_0\rho^2\frac{\pi}{4\mu^2}|\zeta|^2+\bar{z}_2\hspace{.5cm}\mbox{ satisfies }\hspace{.5cm}\hat{z}_2\leq \bar{z}_1-z^*;$$

        \item[(e)] Now, with control $u^*$ and time $t^*$ we have that 
        $$\phi(t^*, (\mathbf{v}^*, \hat{z}_2), u^*)=\phi(t^*, (\mathbf{v}^*,0), u^*)+(\mathbf{0}, \hat{z}_2)=(\mathbf{v}(\rho), z^*)+(\mathbf{0}, \hat{z}_2)=(\mathbf{v}(\rho), z^*+\hat{z}_2);$$

        \item [(f)] With time  $t_3:=\frac{2\mu}{\rho^2|\zeta|^2}(\bar{z}_1-(z^*+\hat{z}_2))\geq 0$ and control $u\equiv \rho$ we get that
        $$\phi(t_3, (\mathbf{v}(\rho), z^*+\hat{z}_2), \rho)=\left(\mathbf{v}(\rho), t_3\frac{\rho^2}{2\mu}|\zeta|^2+(z^*+\hat{z}_2)\right)=(\mathbf{v}(\rho), \bar{z}_1);$$

 \item[(g)] Now, with control $u_1$ and time $t_1$ we get that $\phi(t_1, (\mathbf{v}(\rho), \bar{z}_1), u_1)=(\mathbf{0}, z_1)$, showing the assertion.
       
\end{itemize}

(3) Let us now consider the case where $A$ admits a pair of real eigenvalues. A simple analysis on the characteristic polynomial of $A$, under the assumption that $\tr A=0$ and $\det A\neq 0$, implies necessarily that, on some orthonormal basis $\{\mathbf{e}_1, \mathbf{e}_2\}$, 
    $$A=\left(\begin{array}{cc}
        \mu & 0 \\
       0  & -\mu
    \end{array}\right), \hspace{.5cm}\mbox{ where }\hspace{.5cm}\mu=\sqrt{|\det A|},$$
and hence, the system is given as
$$\left\{\begin{array}{l}
     \dot{x}=\mu x+u\zeta_1\\
     \dot{y}=-\mu y+u\zeta_2\\
     \dot{z}=\frac{u}{2}(\zeta_1y-\zeta_2x)
\end{array}\right., \hspace{.5cm}\mbox{ where }\hspace{.5cm}\zeta=(\zeta_1, \zeta_2), \hspace{.5cm}\mbox{ with }\hspace{.5cm}\zeta_1\zeta_2\neq 0.$$

The diffeomorphism
$$f:\HB \rightarrow \R^3, \hspace{1cm} f(x, y, z)=\left(\frac{\mu}{\zeta_1}x, \frac{\mu}{\zeta_2}y, \frac{\mu^2}{\zeta_1\zeta_2}\left(z+\frac{1}{2}xy\right)\right),$$
conjugates $\Sigma_{\HB}$ to the control-affine system
$$\hspace{-2cm}(\Sigma_{\R^3})\hspace{2cm}\left \{
\begin{array}
[c]{l}%
\dot{x}=\mu(x+u)\\
\dot{y}=\mu(-y+u)\\
\dot{z}=u\mu y
\end{array}
\right.,$$
whose solutions for constant control functions are
$$\phi(t, (\mathbf{
    v}_0, z_0), u)=\left(\rme^{\mu t}(x_0+u)-u, \rme^{-\mu t}(y_0-u)+u, z_0+uy_0(1-\rme^{-\mu t})+u^2(\rme^{-\mu t}+\mu t-1)\right).$$

The projection of $\Sigma_{\R^3}$ onto the first two components is a linear control system on $\R^2$ whose unique control set is given by (see \cite[Theorem 3.6]{DSAyAR})
$$\CC_{\R^2}=-\inner\Omega\times\Omega.$$
    
As a consequence, any control set for $\Sigma_{\R^3}$ has to be contained in the cylinder $\CC_{\R^2}\times\R$. 

Since 
$$\CC_{\R^2}\subset \R\times\Omega \times\R\hspace{.5cm}\mbox{ and }\hspace{.5cm} \phi_{S, {\bf u}}(\R\times\Omega \times\R)\subset \R\times\Omega \times\R,$$
for any $S\geq 0$, it is enough to show that the only control set in $\R\times\Omega \times\R$ is the singleton $\{(0, 0, 0)\}$.

For this, define the function $$G:\R\times\Omega \times\R\rightarrow\R, \hspace{1cm} G(x, y, z):=z+\sigma y+\sigma^2\ln(y-\sigma),$$
where $\sigma<u$ for all $u\in\Omega$. Consider $u\in\Omega$ and use the notation $\phi(S, (\mathbf{v}_0, z_0), u)=(x_S, y_S, z_S)$. 
Then,
$$\frac{d}{dS}G(\phi(S, (\mathbf{v}_0, z_0), u))=\dot{z}_S+\sigma\dot{y}_S+\sigma^2\frac{\dot{y}_S}{y_S-\sigma}=u\mu y_S+\sigma\mu(-y_S+u)+\sigma^2\frac{\mu(-y_S+u)}{y_S-\sigma}$$
$$=\mu\frac{(y_S-\sigma)(uy_S+\sigma(-y_S+u))+\sigma^2(-y_S+u)}{y_S-\sigma}=\mu\frac{(u-\sigma)}{y_S-\sigma}y_S^2\geq 0,$$
showing that $G(\phi(S, (\mathbf{v}_0, z_0), u))> G(\mathbf{v}_0, z_0)$ if $y_0^2+u^2\neq 0$. Arguing as in item (2) of Theorem \ref{bothzeros} we get that the line $z=0$ is a continuum of one-point control sets of $\Sigma_{\R^3}$. Since the conjugation $f$ takes $\ker\DC$ over such line, the result follows.
\end{proof}

\subsection{The case $\det A=0$ and $\tr A\neq 0$}

On the basis that diagonalizes $A$ we get the system

$$\left\{\begin{array}{l}
     \dot{x}=u\zeta_1\\
     \dot{y}=\mu y+u\zeta_2\\
     \dot{z}=\mu z+\frac{u}{2}(\zeta_1y-\zeta_2x)+(\eta_1x+\eta_2 y)
\end{array}\right., \hspace{.5cm}\mbox{ where }\hspace{.5cm}\zeta=(\zeta_1, \zeta_2), \hspace{.5cm}\mbox{ with }\hspace{.5cm}\zeta_1\zeta_2\neq 0,$$
and $\mu=\tr A$. Moreover, 

$$\mbox{ ad-rank condition} \iff \eta_2\neq 0.$$

The diffeomorphism
$$f:\HB\rightarrow\R^3, \hspace{1cm}f(x, y, z)=\left(\frac{\mu}{\zeta_1}x, \frac{\mu}{\zeta_2}y, \frac{\mu^2}{\zeta_1\zeta_2}\left(z+\frac{1}{2}xy+\frac{\eta_1}{\mu\zeta_2}\left(x-y\right)\right)\right)$$
conjugate the system to the control-affine system 

$$(\Sigma_{\R^3})\hspace{1cm}\left\{\begin{array}{l}
     \dot{x}=\mu u\\
     \dot{y}=\mu (y+u)\\
     \dot{z}=\mu \left(z+uy+\frac{\eta_2}{\zeta_1}y\right)
\end{array}\right.$$
%%%%%%%%%%%%%%%%  !!!! It may be necessary to revise this part to \R^2\times\R  !!!!!!!!!!!!!!!!!!!!!!!!
By writing the points in $\R^3$ as $(z, \mathbf{w})\in\R\times\R^2$, the solutions of the previous system, for constant control, are given by 
\begin{equation}
\label{solucao}
    \phi(S, (x, \mathbf{w}), u)=(x+\mu u S, \phi_2(S, \mathbf{w}, u))
\end{equation}
where $\phi_2$ is the solution of the associated system 
$$(\Sigma^A_{\R^2})\hspace{1cm}\left\{\begin{array}{l}
     \dot{y}=\mu (y+u)\\
     \dot{z}=\mu \left(z+uy+\alpha y\right)
\end{array}\right., \hspace{.5cm}\mbox{ and for simplicity we put }\hspace{.5cm}\alpha=\frac{\eta_2}{\zeta_1}.$$
Note that formula (\ref{solucao}) tells us that $\phi$ is linear on the first component. Moreover, systems $\Sigma_{\R^3}$ and $\Sigma_{\R^2}^A$ are conjugated by the canonical projection of $\R^3$ onto the last two components. In particular, we can relate the control set of both systems.

 Due to the previous, let us start by showing the  that the associated system $\Sigma^A_{\R^2}$ admits a unique control set with a nonempty interior. Write
$$A:=\mu\left(\begin{array}{cc}
    1 & 0  \\
    \alpha & 1
    \end{array}\right), \hspace{.5cm}B:=\mu\left(\begin{array}{cc}
    0 & 0  \\
    1 & 0
    \end{array}\right)\hspace{.5cm}\mbox{ and }\hspace{.5cm}C:=\mu\left(\begin{array}{c}
    1  \\
    0
    \end{array}\right).$$
    By considering $\mathbf{w}=(x, y)$ and $A(u):=A+uB$, the system $\Sigma_{\R^2}^A$ can be written, in matricial form, as an affine systems (in the sense of the paper \cite{FCASJS}) 
$$\dot{\mathbf{w}}=A(u)\mathbf{w}+Cu, \hspace{.5cm}u\in\Omega.$$
Since $\det A(u)\neq 0$ for all $u\in\Omega$, the set of equilibria of the system is given by 
$$\mathcal{E}:=\{\mathbf{w}(u):=-A(u)^{-1}Cu,\hspace{.2cm} u\in\Omega\}=\{(-u, u(u+\alpha)), \hspace{.2cm} u\in\Omega\}.$$
Moreover, the vectors
$$B'(u):=C+B\mathbf{w}(u)=\mu\left(\begin{array}{c}
    1  \\
    -u
    \end{array}\right)\hspace{.5cm}\mbox{ and }\hspace{.5cm}A(u)B'(u)=\mu^2\left(\begin{array}{c}
    1  \\
    \alpha
    \end{array}\right),$$
are linearly dependent if and only if $u=-\alpha$. From \cite[Proposition 5.2]{FCASJS}, we conclude that 
$$\OC^+(\mathbf{w}(u))\hspace{.5cm}\mbox{ and }\hspace{.5cm}\OC^-(\mathbf{w}(u))\hspace{.5cm}\mbox{ are open sets for all }\hspace{.5cm}u\in \inner\Omega\setminus \{-\alpha\}.$$
  In particular, for any $u\in \inner\Omega\setminus \{-\alpha\}$ there exists a control set $D_u$ such that $\mathbf{w}(u)\in\inner D_u$.

\begin{proposition}
\label{controlaffine}
    With the previous notations, the affine control system
    $$\dot{\mathbf{w}}=A(u)\mathbf{w}+Cu, \hspace{.5cm}u\in\Omega,$$
admits a control set with a nonempty interior $\CC^A_{\R^2}$ satisfying $\mathcal{E}\subset\overline{\CC^A_{\R^2}}$. Morever, $\CC^A_{\R^2}$ is open if $\mu>0$ and closed if $\mu<0$.
\end{proposition}

\begin{proof}
Let us start by showing that the control sets $D_u$ coincides for $u\in\inner\Omega\setminus\{-\alpha\}$. Since $\mu$ is the only eigenvalue of $A(u)$, it holds that 
$$\forall \mathbf{w}\in\R^2, u\in\Omega \hspace{.5cm}\phi(t, \mathbf{w}, u)\rightarrow \mathbf{w}(u), \hspace{.5cm} \mu t\rightarrow-\infty.$$
Therefore, for $u_1, u_2\in\inner\Omega\setminus\{-\alpha\}$ there exists $t_1, t_2\in\R$ with $\mu t_1, \mu t_2>0$, such that 
$$\phi(t_1, \mathbf{w}(u_2), u_1)\in D_{u_1}\hspace{.5cm}\mbox{ and }\hspace{.5cm}\phi(t_2, \mathbf{w}(u_1), u_2)\in D_{u_2},$$
implying that $D_{u_1}=D_{u_2}$. Therefore, $\CC^A_{\R^2}:= D_u$ for $u\in\inner \Omega\setminus\{-\alpha\}$ is a well defined control set with a nonempty interior of the system satisfying $\mathcal{E}\subset \overline{\CC^A_{\R^2}}$. Moreover, by \cite[Lemma 5.8]{FCASJS}, it holds that $\OC^+(\mathbf{w}(u))=\R^2$ (resp. $\OC^-(\mathbf{w}(u))=\R^2$) if $\mu>0$ (resp. $\mu<0$). Since, 
$$\CC^A_{\R^2}=D_u=\overline{\OC^+(\mathbf{w}(u))}\cap\OC^-(\mathbf{w}(u)),$$
we conclude that $\CC^A_{\R^2}$ is open if $\mu>0$ and closed if $\mu<0$.
\end{proof}

We can now prove the main result of this section.

\begin{theorem}
    If the derivation drift of the linear control system $\Sigma_{\HB}$ satisfies $\det A=0$ and $\tr A\neq 0$ then, up to conjugations, 
 $$\CC_{\HB}=\pi^{-1}(\CC^A_{\R^2}),$$
 is the unique control set of $\Sigma_{\HB}$, where $\pi:\R\times\R^2\rightarrow\R^2$ is the canonical projection onto the last two components.
\end{theorem}

\begin{proof}
In order to show the result, we need to prove the followings:

(1) For any $\mathbf{w}_1, \mathbf{w}_2\in\inner\CC^A_{\R^2}$ there exists $t_0>0$ and $ u\in\UC$ such that 
$$\phi_{t_0, u}(\R\times \{\mathbf{w}_1\})=\R\times \{\mathbf{w}_2\}.$$

In fact, by exact controllability in $\inner\CC^A_{\R^2}$, there exists $t_0>0$ and $ u\in\UC$ such that 
$$\phi_2(t_0, \mathbf{w}_1, u)=\mathbf{w}_2\hspace{.5cm}\implies\hspace{.5cm}\phi_{t_0, u}(\R\times\{\mathbf{w}_1\})\subset\R\times\{\mathbf{w}_2\}.$$

Since $\phi_{t_0, {\bf u}}:\R^3\rightarrow\R^3$ is a diffeomorphism, the equality holds.

(2) For any $u\in\inner\Omega\setminus\left\{-\frac{\eta_2}{\zeta_1}\right\}$ the fiber $\pi^{-1}(\mathbf{w}(u))$ is controllable, where $\mathbf{w}(u)$ is the equilibria of $\Sigma^A_{\R^2}$.

In fact, by Proposition \ref{controlaffine}, the control set $\CC^A_{\R^2}$ satisfy
$$\mathbf{w}(u)\in\inner\CC_{\R^2}, \hspace{.5cm}\forall u\in\inner\Omega\setminus\left\{-\frac{\eta_2}{\zeta_1}\right\}$$

Let then $u_1\in\inner\Omega\setminus\left\{-\frac{\eta_2}{\zeta_1}\right\}$ and assume w.l.o.g. that $\mu u_1>0$, since the other case is analogous. Let $u_2\in \inner\Omega\setminus\left\{-\frac{\eta_2}{\zeta_1}\right\}$ such that $\mu u_2<0$ and consider $t_1, t_2>0$ and $u_1^*, u_2^*\in\UC$ such that 
$$\phi_{t_1, u^*_1}(\R\times \{\mathbf{w}(u_1)\})=\R\times\{\mathbf{w}(u_2)\}\hspace{.5cm}\mbox{ and }\hspace{.5cm}\phi_{t_2, u^*_2}(\R\times \{\mathbf{w}(u_2)\})=\R\times\{\mathbf{w}(u_1)\}.$$
For any given $x, y\in\R$ with $x<y$ let us consider 
$x', y'\in\R$ such that 
$$\phi(t_1, (y, \mathbf{w}(u_1)), u_1^*)=(y', \mathbf{w}(u_2))\hspace{.5cm}\mbox{ and \hspace{.5cm}}\phi(t_2, (x', \mathbf{w}(u_2)), u^*_2)=(x, \mathbf{w}(u)).$$

Moreover, let $z> 0$ such that $z+y'> x'$. A periodic trajectory passing through $(x, \mathbf{w}(u_1))$ and $(y, \mathbf{w}(u_1))$ is constructed as follows:

\begin{enumerate}
    
    \item[(a)] Starting in $(x, \mathbf{w}(u_1))$ with control constant $u_1$ and time $t_3=\frac{z+y-x}{\mu u_1}>0$ we have that
    $$\phi(t_3, (x, \mathbf{w}(u_1)), u_1)=(x+\mu u_1 t_3, \mathbf{w}(u_1))=(z+y, \mathbf{w}(u_1)).$$
    Note that this curve passes through $(y, \mathbf{w}(u_1))$ when $t=\frac{y-x}{\mu u_1}$; 

 \item[(b)] With control $u_1^*$ and time $t_1>0$, go from $(z+y, \mathbf{w}(u_1))$ to the point
 $$\phi(t_1, (z+y, \mathbf{w}(u_1)), u_1^*)=(z, 0)+\phi(t_1, (y, \mathbf{w}(u_1)), u_1^*)=(z, 0)+(y', \mathbf{w}(u_2))=(z+y', \mathbf{w}(u_2));$$

 \item[(c)] With constant control $u_2$ and time $t_4=\frac{x'-y'-z}{\mu u_2}>0$ we go from $(z+y', \mathbf{w}(u_2))$ to
 $$\phi(t_4, (z+y', \mathbf{w}(u_2)), u_2)=(z+y'+\mu u_2t_4, \mathbf{w}(u_2))=(x', \mathbf{w}(u_2));$$

 \item[(d)] Now, with control $u_2^*$ and time $t_2>0$, we go from $(x', \mathbf{w}(u_2))$ to $(x, \mathbf{w}(u_1))$, showing the claim.
\end{enumerate}

(3) $\CC_{\HB}=\pi^{-1}(\CC_{\R^2})$ is a control set.
In fact, using the previous item, one can easily show that  $\pi^{-1}(\inner\CC_{\R^2})$ satisfies properties (1) and (2) in the definition of control sets. In particular, there exists a control set $\CC_{\HB}$ such that $\pi^{-1}(\inner\CC_{\R^2})\subset\CC_{\HB}$. However, by Proposition \ref{CCC} and the previous item, it holds that $\pi^{-1}(\CC_{\R^2})=\CC_{\HB}$ as stated.

\end{proof}

\end{document}